\documentclass{birkau}
\usepackage{amsmath,amssymb,url}
\usepackage{hyperref}
\usepackage{enumitem}
\usepackage{latexsym}
\usepackage[english]{babel}
\usepackage{tikz}
\usepackage{bm}

\renewcommand{\labelenumi}{(\theenumi)}
\numberwithin{equation}{section}

\theoremstyle{plain}
\newtheorem{thm}{Theorem}[section]
\newtheorem{pro}[thm]{Proposition}
\newtheorem{lem}[thm]{Lemma}

\newtheorem{cor}[thm]{Corollary}
\theoremstyle{definition}
\newtheorem{exa}[thm]{Example}
\newtheorem{de}[thm]{Definition}

\newcommand{\N}{\mathbb{N}}

\newcommand{\Con}{\operatorname{Con}}

\newcommand{\Pol}{\operatorname{Pol}}
\newcommand{\RoundPol}{\operatorname{RoundPol}}

\newcommand{\bA}{\mathbf{A}}
\newcommand{\bG}{\mathbf{G}}
\newcommand{\oa}{\mathbf{a}}
\newcommand{\ob}{\mathbf{b}}
\newcommand{\oc}{\mathbf{c}}
\newcommand{\od}{\mathbf{d}}

\newcommand{\tup}[3]{(#1_{#2},\dots,#1_{#3})}
\newcommand{\Iibf}{\bm{\mathcal I}}

\DeclareMathOperator{\dep}{\mathord{d}}
\DeclareMathAlphabet\mathbfcal{OMS}{cmsy}{b}{n}

\DeclareMathOperator{\Cc}{\mathord{C}}

\begin{document}

\title{(Super)nilpotency and solvability in algebras with zero}
\corrauthor[I.\ \DJ ur\dj ev]{Ivana \DJ ur\dj ev}
\address{Mathematical Institute
of the Serbian Academy of Sciences and Arts \\
Kneza Mihaila 36\\
11000 Belgrade\\
Serbia}
\email{ivana.djurdjev@mi.sanu.ac.rs}

\author[N.\ Mudrinski]{Neboj\v{s}a Mudrinski}
\address{Department of Mathematics and Informatics\\
Faculty of Sciences\\
University of Novi Sad\\
21000 Novi Sad\\
Serbia}
\email{nmudrinski@dmi.uns.ac.rs}


\begin{abstract}
An algebra with zero is an algebra that contains an element 0, called zero, such that all fundamental operations have the property that the result is zero whenever at least one of the arguments takes the value 0. We introduce the notion of the higher commutator of ideals and Rees congruences for such algebras. It is shown that the commutator of Rees congruences is equal to the Rees congruence of the commutator of the corresponding ideals. We apply this result to provide characterisations of $n$-supernilpotency and $n$-nilpotency in algebras with zero, as well as $n$-solvability in groupoids with zero.
\end{abstract}

\maketitle

\section{Motivation and results}

In this manuscript, we study the properties of an algebra defined via the term condition commutator: $n$-supernilpotency, $n$-nilpotency and $n$-solvability for $n\in \mathbb{N}$. The term condition commutator has been defined as a binary operation $[\bullet,\bullet]$ on the congruence lattice of an algebra, first in congruence permutable varieties, and later in congruence modular varieties; see \cite{FreeMcKen}. It has been used to generalise the notions of abelianness, nilpotency and solvability, due to their importance in group theory (see Definitions \ref{d:Abel}, \ref{d:nilpotent} and \ref{d:n-sol}). In 2001, A. Bulatov generalised it by defining the commutator $[\bullet,\dots,\bullet]$ of any finite length; see \cite{Bulatov}. This concept enables the definition of the supernilpotent property $[1_A,\dots,1_A]=0_A$, for the full congruence $1_A$ and equality relation $0_A$ of an algebra $\bA$ (see Definition \ref{d:supernil}). Two main questions arise naturally. First, how are the mentioned properties connected? Second, can we find a characterisation of algebras with any of these properties? Considering the first question, from the basic properties of commutators, we easily obtain that abelianness implies $n$-nilpotency and $n$-supernilpotency, and that $n$-nilpotency implies $n$-solvability. However, the relationship between $n$-nilpotency and $n$-supernilpotency is not so straightforward and is the subject of a number of articles. $n$-supernilpotency implies $n$-nilpotency in congruence permutable varieties (as shown in \cite{AichMudr}), as well as in congruence modular varieties and in Taylor algebras (proved by Moorhead \cite{Moorhead2}) and in all finite algebras (proved by Kearnes and Szendrei \cite{KearSzen}). This implication is not always true, as Moore and Moorhead constructed an $n$-supernilpotent infinite algebra which is not even solvable (see \cite{MooreMoorhead}). Nevertheless, it is well known that $n$-supernilpotency and $n$-nilpotency coincide in groups and rings. Recently, it has been proved that $n$-nilpotency and $n$-supernilpotency also coincide in semigroups with zero \cite{RadMudr1} and in semirings with zero \cite{MilicaMudr}. Since such algebras are outside the congruence modular varieties and their generalisations, we may wonder: what is the property of an algebra that forces $n$-nilpotency and $n$-supernilpotency to coincide? Is the existence of the zero element such a condition? Our investigation is devoted to that question. For that purpose, we introduce algebras with zero in Definition \ref{awz} (the simplest examples are groupoids with zero). For them, we adopted the concept of Rees congruences from semigroups. In this manuscript, we study the commutator of these congruences and prove that $n$-supernilpotency implies $n$-nilpotency in algebras with zero; see Proposition \ref{supernilpnilp}. However, $n$-supernilpotency does not need to coincide with $n$-nilpotency, as shown in Example \ref{counterexample}. Most importantly, we provide characterisations of $n$-supernilpotent algebras with zero in Theorem \ref{p:supernil}, and of $n$-nilpotent algebras with zero in Theorem \ref{t:n-nilpot}, as well as a characterisation of $n$-solvable grupoids with zero, in Theorem \ref{t:nsolGroupoid}.

\section{On commutators in arbitrary algebras}

We begin by introducing the necessary notions, our primary goal being to describe the terms: the commutator of ideals, the commutator of congruences, $n$-nilpotency, $n$-supernilpotency, and $n$-solvability.
Recall the notation and definitions from \cite{McMc} for congruences, terms, and polynomial terms in universal algebras. We will omit writing the superscript
${\bA}$ for the polynomial $p^{\bA}(x_1,\dotsc,x_n)$ on $\bA$ corresponding to the polynomial term $p$, unless necessary. The set of all $k$-ary polynomials corresponding to a polynomial term containing all the $k$ arguments over an algebra $\bA$ will be denoted $\Pol^{*}_{k}(\bA)$. Furthermore, $\Pol^{*}(\bA)=\bigcup_{k\in\N}\Pol^{*}_{k}(\bA)$.


\begin{de}\label{Ideal}
 A non-empty set $I\subseteq A$ is an \textit{ideal} of an algebra $\bA=(A,F)$ if, for all $f\in F$ of arity $n\in\N$, we have $f\tup{x}{1}{n}\in I$, whenever there is an $i\in\{1,\dots,n\}$ such that $x_i\in I$.
\end{de}

The next proposition is an easy corollary of this definition.

\begin{pro}\label{idealsandpolynomials}
Let $I$ be an ideal of an algebra $\bA$, and let $k\in\N$ and $p\in\Pol^{*}_{k}(\bA)$. If $\tup{a}{1}{k}\in A^k$ contains at least one element of $I$, we have $p\tup{a}{1}{k}\in I$.
\end{pro}

 As in the case of semigroups and rings, ideals are domains of subalgebras, and the opposite is not true, in general. Next, we introduce a notion from semigroup theory (see \cite{Howie}). It is straightforward to check that $\rho_I=I\times I\cup\{(a,a)\,|\,a\in A\backslash I\}$ is a congruence of $\bA$, for an ideal $I$. Thus, we give the following definition.

\begin{de}\label{Reescongruence}
 Let $I$ be an ideal of an algebra $\bA$. Then, $\rho_I$ is the \textit{Rees congruence} on $\bA$ that corresponds to the ideal $I$.
\end{de}

With all the necessary preparations completed, we may define the higher commutator of ideals in an arbitrary algebra and examine its properties. Since this notion is trivial for algebras that do not contain operations of arity $2$ or higher (see Definition \ref{d:commutatorIdeals} and Example \ref{unary}), we consider only those algebras that have at least one fundamental operation of arity at least two. Unless otherwise stated, we assume this condition to hold for the remainder of this paper. It is a well-known result of Sierpi\'nski that such algebras have polynomials of arbitrary finite arity. 

\begin{de}\label{d:commutatorIdeals} Let $\bA$ be an algebra. For $k\in\N$ and ideals $I_{1},\ldots, I_{k}$ in $\bA$, we define the \textit{commutator of} $I_{1},\ldots I_{k}$ as
$$[I_{1},\ldots,I_{k}]=\{p(i_{1},\ldots,i_{k}): p\in\Pol^{*}_{k}(\bA), \text{ and } i_{1}\in I_{1},\ldots, i_{k}\in I_{k}\}.$$
\end{de}

\begin{pro}\label{p:comm_is_ideal} Let $\bA$ be an algebra. If $k\in\N$ and $I_{1},\ldots, I_{k}$ are ideals of $\bA$, then $[I_{1},\ldots,I_{k}]$ is an ideal of $\bA$.
\end{pro}
\begin{proof}
Since $I_{i}\neq \emptyset$ for $i=1,\ldots,k$, for elements $i_1\in I_1,\dots,i_k\in I_k$ and an arbitrary polynomial $p\tup{x}{1}{k}\in\Pol^{*}_{k}(\bA)$, we have $p\tup{i}{1}{k}\in[I_{1},\ldots,I_{k}]$. Therefore, $[I_{1},\ldots,I_{k}]\neq\emptyset$. Now, let $i\in[I_{1},\ldots,I_{k}]$, and let $1\leq j\leq k$. Further, let $f$ be a fundamental operation of arity $n\in\N$, and let $a_1,\dots,a_{j-1},a_{j+1},\dots,a_n\in A$. From Definition \ref{d:commutatorIdeals}, we have $i=p(i_{1},\ldots,i_{k})$, for some $p\in\Pol^{*}_{k}(\bA)$ and elements $i_{t}\in I_{t}$ for $t\in\{1,\dots,k\}$.
Since $f(a_1,\dots,a_{j-1},p(x_{1},\ldots,x_{k}),a_{j+1},\dots,a_n)\in\Pol^{*}_{k}(\bA)$, by Definition \ref{d:commutatorIdeals} we may conclude that
$$
f(a_1,\dots,a_{j-1},i,a_{j+1},\dots,a_n)=f(a_1,\dots,a_{j-1},p(i_{1},\ldots,i_{k}),a_{j+1},\dots,a_n)
$$
is an element of $[I_{1},\ldots,I_{k}]$. Hence, we have proved that $[I_{1},\ldots,I_{k}]$ is an ideal of $\bA$, by Definition \ref{Ideal}.
\end{proof}

The following statements are true in any algebra, as proved in \cite[Section 2, Proposition 1.3]{AichMudr}.

\begin{pro}\label{HC} Let $\bA$ be an algebra and let $k\in\N$. For Rees congruences $\rho_{I_{0}},\rho_{I_{1}},\ldots,\rho_{I_{k}},\rho_{J_{1}},\ldots,\rho_{J_{k}}$ of $\bA$, we have
\begin{description}
    \item[(HC1)] $[\rho_{I_{1}},\ldots,\rho_{I_{k}}]\leq \rho_{I_{1}}\wedge\cdots\wedge \rho_{I_{k}}$;
    \item[(HC2)] if $\rho_{J_{1}}\leq \rho_{I_{1}},\ldots, \rho_{J_{k}}\leq \rho_{I_{k}}$, then $[\rho_{J_{1}},\ldots,\rho_{J_{k}}]\leq [\rho_{I_{1}},\ldots, \rho_{I_{k}}]$;
    \item[(HC3)] $[\rho_{I_{0}},\rho_{I_{1}},\ldots,\rho_{I_{k}}]\leq [\rho_{I_{1}},\ldots,\rho_{I_{k}}]$.
\end{description}
\end{pro}

In our next step, we define the commutator of congruences, as in \cite{Bulatov}. This is followed by the definitions of abelianness, $n$-nilpotency, $n$-supernilpotency and $n$-solvability for an arbitrary algebra. We aim to find the connection between the two notions of commutators and then describe these properties via that connection.

\begin{de}\label{d:C} Let $\bA$ be an algebra, $k\in\N$, and $\alpha_{1},\ldots,\alpha_{k},\beta, \delta\in\Con(\bA)$. We say that congruences $\alpha_{1},\ldots,\alpha_{k}$ on $\bA$ \textit{centralize} $\beta$ \textit{modulo} $\delta$ (denoted by $\Cc(\alpha_{1},\ldots,\alpha_{k},\beta;\delta)$), if for each polynomial $p\in\Pol(\bA)$ of arity $n_{1}+\cdots +n_{k}+m$ where $n_{1},\ldots,n_{k},m\in \N$, and any $\mathbf{a}_{1},\mathbf{b}_{1}\in A^{n_{1}},\ldots,\mathbf{a}_{k},\mathbf{b}_{k}\in A^{n_{k}}, \mathbf{c},\mathbf{d}\in A^{m}$ which satisfy $\oc\, \beta\, \od$ and $\oa_j\, \alpha_{j}\, \ob_{j}$ for all $j\in\{1,\ldots,k\}$,  the following relations
    $$
    \begin{aligned}
        p(\mathbf{a}_{1},\mathbf{a}_{2},\ldots,\mathbf{a}_{k-2},\mathbf{a}_{k-1},\mathbf{a}_{k},\mathbf{c})\ &\delta\  p(\mathbf{a}_{1},\mathbf{a}_{2},\ldots,\mathbf{a}_{k-2},\mathbf{a}_{k-1},\mathbf{a}_{k},\mathbf{d}),\\
        p(\mathbf{a}_{1},\mathbf{a}_{2},\ldots,\mathbf{a}_{k-2},\mathbf{a}_{k-1},\mathbf{b}_{k},\mathbf{c})\ &\delta\  p(\mathbf{a}_{1},\mathbf{a}_{2},\ldots,\mathbf{a}_{k-2},\mathbf{a
        }_{k-1},\mathbf{b}_{k},\mathbf{d}),\\
        &\cdots\\
        p(\mathbf{b}_{1},\mathbf{a}_{2},\ldots,\mathbf{a}_{k-2},\mathbf{a}_{k-1},\mathbf{a}_{k},\mathbf{c})\ &\delta\  p(\mathbf{b}_{1},\mathbf{a}_{2},\ldots,\mathbf{a}_{k-2},\mathbf{a}_{k-1},\mathbf{a}_{k},\mathbf{d}),\\
        p(\mathbf{b}_{1},\mathbf{a}_{2},\ldots,\mathbf{a}_{k-2},\mathbf{a}_{k-1},\mathbf{b}_{k},\mathbf{c})\ &\delta\  p(\mathbf{b}_{1},\mathbf{a}_{2},\ldots,\mathbf{a}_{k-2},\mathbf{a}_{k-1},\mathbf{b}_{k},\mathbf{d}),\\
        &\cdots\\
        p(\mathbf{b}_{1},\mathbf{b}_{2},\ldots,\mathbf{b}_{k-2},\mathbf{b}_{k-1},\mathbf{a}_{k},\mathbf{c})\ &\delta\  p(\mathbf{b}_{1},\mathbf{b}_{2},\ldots,\mathbf{b}_{k-2},\mathbf{b}_{k-1},\mathbf{a}_{k},\mathbf{d})
    \end{aligned}
    $$
    together imply $p(\mathbf{b}_{1},\ldots,\mathbf{b}_{k},\mathbf{c})\ \delta\  p(\mathbf{b}_{1},\ldots,\mathbf{b}_{k},\mathbf{d})$.

\end{de}

\begin{de}\label{d:commCong} Let $\bA$ be an algebra and let $k\geq 2$. For $\alpha_{1},\ldots,\alpha_{k}\in\Con(\bA)$, the \textit{$k$-ary commutator of} $\alpha_{1},\ldots,\alpha_{k}$ is the smallest congruence $\delta$ on $\bA$ such that $\Cc(\alpha_{1},\ldots,\alpha_{k};\delta)$. This congruence is denoted by $[\alpha_{1},\ldots,\alpha_{k}]$.
\end{de}

\begin{de}\label{d:Abel}
An algebra $\bA$ is abelian if $[1_{A},1_{A}]=0_{A}$.
\end{de}

\begin{exa}\label{unary}
Let $\bA=(A,F)$ be an algebra such that all operations in $F$ are at most unary. Here, we have $\Pol^{*}(\bA)=\Pol^{*}_{1}(\bA)$. From Definition \ref{d:C} (for $k=1$) we can easily see that for any congruence $\alpha$ on $\bA$ we have $\Cc(1_A,1_A;\alpha)$. Therefore, $[1_A,1_A]=0_A$ by Definition \ref{d:commCong} (for $k=2$). Using the properties (HC1), (HC2) and (HC3) from Proposition \ref{HC}, we may conclude that this algebra also satisfies the next three properties.
\end{exa}

For the following three definitions, let $n\in\mathbb{N}$.

\begin{de}\label{d:nilpotent} An algebra $\bA$ is \textit{$n$-nilpotent}, if $\underbrace{[1_{A},\ldots[}_{n}1_{A},1_{A}]\ldots]=0_{A}$.
\end{de}

\begin{de}\label{d:supernil} An algebra $\bA$ is \textit{$n$-supernilpotent}, if $[\ \underbrace{1_{A},\ldots,1_{A}}_{n+1}\ ]=0_{A}$.
\end{de}

For a congruence $\alpha$ on an algebra $\bA$, let $[\alpha]^{(1)}=[\alpha,\alpha]$ and $[\alpha]^{(n+1)}=[[\alpha]^{(n)},[\alpha]^{(n)}]$, for $n\in\N$.
\begin{de}\label{d:n-sol} An algebra $\bA$ is \textit{$n$-solvable} if $[1_{A}]^{(n)}=0_{A}$.
\end{de}

 The proof of the next lemma is analogous to the proof of \cite[Lemma 2.8]{RadMudr1}. Note that the proof uses the term semantic dependence:
 \begin{de}\label{d:semDep} We say that, for $n\in \mathbb{N}$ and an algebra $\bA$, an $n$-ary polynomial $p$ \textit{semantically depends} on its i-th argument if there exist $a,b\in A$ and $(x_{1},\ldots, x_{n})\in A^{n}$ such that
$$p(x_{1},\ldots,x_{i-1},a,x_{i+1},\ldots,x_{n})\neq p(x_{1},\ldots,x_{i-1},b,x_{i+1},\ldots,x_{n}).$$
 \end{de}
 It is easily seen that any polynomial which semantically depends on all its arguments belongs to the set $\Pol^{*}_{n}(A)$. For this reason, the same proof works in our case as well.

\begin{lem}\label{Lema1}
 Let $k\geq 2$ and let $\rho_{I_1},\dotsc,\rho_{I_k}$ be Rees congruences on an arbitrary algebra $\bA=(A,F)$. Then the condition $\Cc(\rho_{I_1},\dots,\rho_{I_k};\eta)$ is true for all congruences $\eta$ such that $\rho_{[I_1,\dots,I_k]}\leq\eta$.
\end{lem}

\section{Algebras with zero}

In an algebra with a zero element, the converse of Lemma \ref{Lema1} also holds. Thus, we may prove Proposition \ref{p:THEthm} and Corollary \ref{c:isomAlg}, which will give us a way to examine $n$-nilpotency, $n$-supernilpotency, and $n$-solvability in terms of ideals.

\begin{de}\label{awz}
A universal algebra $\bA=(A,F)$ is an \emph{algebra with zero} if there is a constant $0\in F$ such that for any operational symbol $f\in F$ of arity $n\in\N$, we have $f^{\bA}\tup{x}{1}{n}=0$ whenever there exists an $i\in\{1,\dots,n\}$ such that $x_i=0$.
\end{de}

Obvious examples are semigroups with zero and groupoids with zero. These topics are investigated in \cite{RadMudr1} and in Section \ref{grupoids}, but we continue here with the general case. The next statement is a corollary of Proposition \ref{idealsandpolynomials}, since $\{0\}$ is an ideal in any algebra with zero.

\begin{pro}\label{polynomialzero}
Let $\bA$ be an algebra with zero $0$. Then, for each polynomial $p\in\Pol^{*}_{n}(\bA)$ where $n\in\N$, we have $p^{\bA}\tup{x}{1}{n}=0$ whenever there exists an $i\in\{1,\dots,n\}$ such that $x_i=0$.
\end{pro}

In order to make a distinction between the zero polynomial $p(x)=0$ and a polynomial which evaluates to zero for every possible input in an algebra, we introduce a concept from Ring theory.

\begin{de}\label{vanishing} Let $\bA$ be an algebra with zero $0$. We say that a polynomial $p$ of arity $n\in\mathbb{N}$ is a \textit{vanishing polynomial} in $\bA$, if $p(a_{1},\ldots,a_{n})=0$ for any $n$-tuple $(a_{1},\ldots,a_{n})\in A^{n}$.
\end{de}

Note that, in a non-trivial algebra with zero, a non-vanishing polynomial semantically depends (see Definition \ref{d:semDep}) on its $i$-th argument if and only if it appears in the corresponding polynomial term. It follows that, in an algebra $\bA$ with zero, for $n\in\N$, the set $\Pol^{*}_{n}(\bA)$ contains all non-vanishing $n$-ary polynomials which semantically depend on all their arguments, as well as the vanishing $n$-ary polynomials.

It is obvious that in algebras with zero, the zero element belongs to any ideal. Therefore, the simplest examples of ideals are the underlying set $A$ and the zero ideal $\{0\}$. Due to the properties of the zero element discussed above, we can prove Lemma \ref{Lema2}, which is the converse of Lemma \ref{Lema1}. From these two lemmas, in Proposition \ref{p:THEthm}, we obtain the connection between the commutator of Rees congruences and the Rees congruence of the commutator of the corresponding ideals. The proof of Lemma \ref{Lema2} is analogous to the proof of \cite[Lemma 3.1]{RadMudr1}, using Proposition \ref{polynomialzero}.

\begin{lem}\label{Lema2} Let $k\geq 2$ and let $I_{1},\ldots, I_{k}$ be ideals in an algebra $\bA$ with zero. For the corresponding Rees congruences $\rho_{I_{1}},\ldots \rho_{I_{k}}$, and $\theta\in\Con(\bA)$ such that $\Cc(\rho_{I_{1}},\ldots, \rho_{I_{k}};\theta)$, we have $\rho_{[I_{1},\ldots,I_{k}]}\subseteq \theta$.
\end{lem}

Proposition \ref{p:THEthm} follows as a direct consequence of Lemmas \ref{Lema1} and \ref{Lema2}.

\begin{pro}\label{p:THEthm} Let $k\geq 2$ and let $I_{1},\ldots,I_{k}$ be ideals of an algebra with zero $\bA$. For the corresponding Rees congruences $\rho_{I_{1}},\ldots,\rho_{I_{k}}$ we have
\begin{equation*}
[\rho_{I_{1}},\ldots,\rho_{I_{k}}]=\rho_{[I_{1},\ldots,I_{k}]}.
\end{equation*}
\end{pro}

Observe that, in an algebra $\bA$ with zero, we have $\rho_{\{0\}}=0_A$ and $\rho_{A}=1_A$. We denote by $\Iibf(\bA)$ the algebra of all ideals of $\bA$, with operations
$$\{\cap,\cup\}\cup\{[\ \underbrace{\cdot,\ldots,\cdot}_{k}\ ]:k\geq 2\},$$
where $\cap$ and $\cup$ denote the binary operations of intersection and union. Also, for $k\in\N$, $[I_{1},\ldots,I_{k}]$ denotes the $k$-ary commutator of ideals $I_{1},\ldots,I_{k}$, introduced in Definition \ref{d:commutatorIdeals}. Further, let $\mathbf{ RCon}(\bA)$ denote the algebra of all Rees congruences corresponding to ideals in $\bA$, with operations
$$\{\wedge,\vee\}\cup\{[\ \underbrace{\cdot,\ldots,\cdot}_{k}\ ]:k\geq 2\},$$
where $\wedge$ and $\vee$ denote the meet and join of congruences, and for $k\in\N$, $[\rho_{I_{1}},\ldots,\rho_{I_{k}}]$ is the $k$-ary commutator of Rees congruences $\rho_{I_{1}},\ldots,\rho_{I_{k}}$, from Definition \ref{d:commCong}. The following statement can be proved directly as in \cite[Corollary 3.2]{RadMudr1}.

\begin{cor}\label{c:isomAlg} Let $\bA$ be an algebra with zero. Then $\Iibf(\bA)\cong\mathbf{RCon}(\bA)$.
\end{cor}

 Using this connection, we will investigate the properties of commutators of ideals. As usual, we use $S_k$ for the group of all permutations of the set $\{1,\dotsc,k\}$, where $k\in\N$.

\begin{lem}\label{l:PropIdeals} Let $\{J_{l}:l\in L\}$ be a family of ideals of an algebra $\bA$ with zero. Further, let $k,m\in\N$ and $j\in\{1,\ldots,k\}$, and let $I_{0},I_{1},\ldots,I_{k+m},J_{1},\ldots,J_{k}$ be ideals of $\bA$, as well. Then, we have
\begin{description}
    \item[(ACI1)] $[I_{1},\ldots,I_{k}]\subseteq I_{1}\cap\cdots\cap I_{k}$;
    \item[(ACI2)] if $J_{1}\subseteq I_{1},\ldots, J_{k}\subseteq I_{k}$, then $[J_{1},\ldots,J_{k}]\subseteq [I_{1},\ldots, I_{k}]$;
    \item[(ACI3)] $[I_{0},I_{1},\ldots,I_{k}]\subseteq [I_{1},\ldots,I_{k}]$;
    \item[(ACI4)] $[I_{\pi(1)},\ldots,I_{\pi(k)}]=[I_{1},\ldots,I_{k}]$ for all $\pi\in S_k$;
    \item[(ACI7)] $
    \bigcup_{l\in L}[I_{1},\ldots,I_{j-1},J_{l},I_{j+1},\ldots, I_{k}]=\\
    =[I_{1},\ldots,I_{j-1},\bigcup_{l\in L}J_{l},I_{j+1},\ldots, I_{k}]$;
    \item[(ACI8)] $[I_{1},\ldots,I_{k},[I_{k+1},\ldots,I_{k+m}]]\subseteq[I_{1},\ldots,I_{k+m}]$.
\end{description}
\end{lem}
\begin{proof} Properties \textbf{(ACI1)}--\textbf{(ACI3)} follow from Corollary \ref{c:isomAlg} and Proposition \ref{HC}. For \textbf{(ACI4)}, we observe that a permutation of the arguments in a polynomial term results in a polynomial term depending on the same number of variables. The proof for \textbf{(ACI7)} is divided: for $(\subseteq)$, we use \textbf{(ACI2)}, and for $(\supseteq)$ we use Definition \ref{d:commutatorIdeals} directly. Now, let us prove \textbf{(ACI8)}. Let $a\in[I_{1},\ldots,I_{k},[I_{k+1},\ldots,I_{k+m}]]$. By Definition \ref{d:commutatorIdeals}, we have $a=p\tup{a}{1}{k+1}$ for $p\in\Pol^{*}_{k+1}(\bA)$ such that $a_1\in I_1,\dots,a_k\in I_k$ and $a_{k+1}\in[I_{k+1},\ldots,I_{k+m}]$. Hence, there exists a $q\in\Pol^{*}_{m}(\bA)$ such that $a_{k+1}=q\tup{i}{1}{m}$ for $i_1\in I_{k+1},\dots,i_{m}\in I_{k+m}$. Now, the polynomial $p(x_1,\dots,x_k,q(x_{k+1},\ldots,x_{k+m}))$ is in $\Pol^{*}_{k+m}(\bA)$, and therefore $a=p(a_1,\dots,a_k,q\tup{i}{1}{m})\in[I_{1},\ldots,I_{k+m}]$.
\end{proof}

\begin{lem}\label{HC4-8} Let $\{\rho_{J_{l}}:l\in L\}$ be a family of Rees congruences in an algebra $\bA$ with zero. Further, let $k,m\in\N$, and $j\in\{1,\ldots,k\}$. For Rees congruences $\rho_{I_{0}},\rho_{I_{1}},\ldots,\rho_{I_{k}},\rho_{J_{1}},\ldots,\rho_{J_{k}}$ on $\bA$, we have
\begin{description}
    \item[(HC4)] $[\rho_{I_{\pi(1)}},\ldots,\rho_{I_{\pi(k)}}]=[\rho_{I_{1}},\ldots,\rho_{I_{k}}]$ for any $\pi\in S_k$;
    \item[(HC5)] $\Cc(\rho_{I_{1}},\ldots,\rho_{I_{k}};\delta)$ iff $[\rho_{I_{1}},\ldots,\rho_{I_{k}}]\leq \delta$;
    \item[(HC7)] $\bigvee_{l\in L}[\rho_{I_{1}},\ldots,\rho_{I_{j-1}},\rho_{J_{l}},\rho_{I_{j+1}},\ldots, \rho_{I_{k}}]=\\ =[\rho_{I_{1}},\ldots,\rho_{I_{j-1}},\bigvee_{l\in L}\rho_{J_{l}},\rho_{I_{j+1}},\ldots, \rho_{I_{k}}]$;
    \item[(HC8)] $[\rho_{I_{1}},\ldots,\rho_{I_{k}},[\rho_{I_{k+1}},\ldots,\rho_{I_{k+m}}]]\leq[\rho_{I_{1}},\ldots,\rho_{I_{k+m}}]$.
\end{description}
\end{lem}
\begin{proof} Statement \textbf{(HC5)} follows directly from Lemma \ref{Lema1}, Definition \ref{d:commCong} and Proposition \ref{p:THEthm}. Statements \textbf{(HC4)}, \textbf{(HC7)} and \textbf{(HC8)} follow from Lemma \ref{l:PropIdeals} (properties \textbf{(ACI4)}, \textbf{(ACI7)}, and \textbf{(ACI8)}, respectively) and the connection established in Corollary \ref{c:isomAlg}.
\end{proof}

Now, for $n\in\mathbb{N}$, we may give the characterisation of $n$-supernilpotent algebras with a zero element. As usual, for $m\in\N$, an algebra $\bA$, and $p\in\Pol^{*}_{m+1}(\bA)$ we denote
$$p(A^{m+1})=\{p(a_{1},\ldots,a_{m+1}) \mid a_{1},\ldots,a_{m+1}\in A\}.$$

\begin{thm}\label{p:supernil}
    Let $\bA$ be an algebra with zero $0$, and let $n\in\N$. Then, $\bA$ is $n$-supernilpotent if and only if all $p\in\Pol^{*}_{n+1}(\bA)$ are vanishing polynomials.
\end{thm}
\begin{proof} By Definition \ref{d:supernil}, $\bA$ is $n$-supernilpotent if $[\ \underbrace{1_{A},\ldots,1_{A}}_{n+1}\ ]=0_{A}$, which holds if and only if $[\ \underbrace{\rho_{A},\ldots,\rho_{A}}_{n+1}\ ]=\rho_{\{0\}}$. By Proposition \ref{p:THEthm}, this is equivalent to
$\rho_{[\, \underbrace{\scriptstyle A,\ldots,A}_{n+1}\,]}=\rho_{\{0\}},$
which is true if and only if $[\ \underbrace{A,\ldots,A}_{n+1}\ ]=\{0\}$, by Corollary \ref{c:isomAlg}. This means that $p(A^{n+1})=\{0\}$ for any $p\in\Pol^{*}_{n+1}(\bA)$, by Definition \ref{d:commutatorIdeals}. The statement follows using Definition \ref{vanishing}.
\end{proof}

We will derive similar (though not that easy to check) conditions for $n$-nilpotency and $n$-solvability. But before that, we may prove the following.

\begin{pro}\label{supernilpnilp}
  Let $\bA$ be an algebra with zero and let $n\in\N$. If $\bA$ is $n$-supernilpotent, then $\bA$ is $n$-nilpotent.
\end{pro}
\begin{proof}
The statement follows directly from Definitions \ref{d:nilpotent} and \ref{d:supernil} and Proposition \ref{HC4-8} (part (HC8)).
\end{proof}

In Definition \ref{d:n-sol}, we use the congruence $[1_{A}]^{(n)}$ in order to define $n$-solvability. Here, we introduce an ideal which is an analogue of that congruence, in order to determine the condition for $n$-solvability.

\begin{de}\label{d:P_n}
    In an algebra $\bA$ with a zero element, let $[A]^{(1)}=[A,A]$, and for $n\geq 1$ let $[A]^{(n+1)}=[[A]^{(n)},[A]^{(n)}]$.
\end{de}

The following statement can be proved by induction.

\begin{lem}\label{l:rho_P_n}
    Let $\bA$ be an algebra with zero. For $n\in\mathbb{N}$, we have $[\rho_{A}]^{(n)}=\rho_{[A]^{(n)}}$.
\end{lem}

\begin{pro}\label{t:n-sol} Let $\bA$ be an algebra with zero $0$, and let $n\in\mathbb{N}$. Then, $\bA$ is $n$-solvable if and only if $[A]^{(n)}=\{0\}$.
\end{pro}
\begin{proof} By Definition \ref{d:n-sol}, $\bA$ is $n$-solvable if $[1_{A}]^{(n)}=0_{A}$, which is equivalent to $[\rho_{A}]^{(n)}=\rho_{\{0\}}$. By Lemma \ref{l:rho_P_n}, we have $[\rho_{A}]^{(n)}=\rho_{[A]^{(n)}}$, so $A$ is $n$-solvable if and only if $[A]^{(n)}=\{0\}$.
\end{proof}

Analogously, we can prove the following proposition.

\begin{pro}\label{p:n-nil} Let $\bA$ be an algebra with zero $0$, and let $n\in\mathbb{N}$. Then, $\bA$ is $n$-nilpotent if and only if $\underbrace{[A,\ldots[}_{n}A,A]\ldots]=\{0\}$.
\end{pro}

Since the conditions in the previous two statements are not stated in terms of polynomials (as is the case in Proposition \ref{p:supernil}), our next goal is to find alternative descriptions for $n$-nilpotency or $n$-solvability in an algebra with zero to make it easier to check the relationships among $n$-supernilpotency, $n$-nilpotency and $n$-solvability.

\section{Nilpotency in algebras with zero}\label{grupoids}

In this section, we assume that our algebras with zero have no unary fundamental operations. To investigate $n$-nilpotency, we need to determine the elements of the sets $[A,A],[A,[A,A]],[A,[A,[A,A]]],\ldots$. These are described in Lemmas \ref{l:G^2} and \ref{l:pi_n}, but the core of the analysis is conducted in the proof of Lemma \ref{l:[G,I]}.



\begin{lem}\label{l:[G,I]} Let $\bA=(A,F)$ be an algebra with zero, and let $I$ be an ideal of $\bA$. Then,
$$[A,I]=[I,A]=\bigcup_{f\in F}(f(I,A,\dots,A)\cup f(A,I,\dots,A)\cup\ldots\cup f(A,\dots,A,I)).$$
\end{lem}
\begin{proof} The equality $[A,I]=[I,A]$ follows from part \textbf{(ACI4)} of Lemma \ref{l:PropIdeals}. Let us prove the remaining equality. To prove $(\subseteq)$, let $c\in [A,I]$. By Definition \ref{d:commutatorIdeals}, $c=p(a,i)$ for some $p\in\Pol^{*}_{2}(A)$ and elements $a\in A$ and $i\in I$. We prove the inclusion by induction on the complexity of the polynomial term of $p$. For the base case, we assume that $p$ contains only one operation symbol. Suppose without loss of generality that $p(x_1,x_2)=f(x_1,x_2,a_1,\dots,a_{k-2})$ for some $f\in F$ and elements $a_1,\dots,a_{k-2}\in A$.
Thus, $p(a,i)=f(a,i,a_1,\dots,a_{k-2})\in f(A,I,A,\dots,A)$ and the basis holds. For the inductive hypothesis, we assume that the statement is true for any polynomial term with fewer than $n$ operation symbols. Now, let $p$ have exactly $n$ operation symbols. Then $p(a,i)=f\tup{p}{1}{k}(a,i)$ for some polynomial terms $p_{1},\dots,p_{k}$. Here, $p_1,\dots,p_k$ have fewer than $n$ operation symbols and at least one of them depends on $x_2$. Without loss of generality, we assume that $p_1$ depends on $x_2$. By the inductive hypothesis, there exists a $g\in F$ such that $p_1(a,i)\in g(A,\dots,A,I,A,\dots,A)\subseteq I$ (the inclusion follows from Definition \ref{Ideal}). Clearly, $p_2(a,i),\dots,p_k(a,i)\in A$. Hence, we have $p(a,i)\in  f(I,A,\dots,A)$. Thus, we have proved the induction step.
To prove $(\supseteq)$, let $f\in F$ be of arity $n\geq 2$, let $a_1,\dots,a_{n-1}\in A$ and $i\in I$. Then, for $p(x_1,x_2)=f(x_1,x_2, a_2,\dots,a_{n-1})$, we have $p\in\Pol^{*}_{2}(\bA)$ and $p(i,a_1)=f(i,a_1,\dots,a_{n-1})$. Since $p(i,a_1)\in[I,A]$, we obtain $f(i,a_1,\dots,a_{n-1})\in[I,A]$. Therefore, $f(I,A,\dots,A)\subseteq[I,A]$. The proof for the remaining sets is analogous.
\end{proof}

\begin{lem}\label{l:G^2}
    In an algebra $\bA=(A,F)$ with zero, we have
    $$[A,A]=\bigcup_{f\in F}f(A,\dots,A).$$
\end{lem}
\begin{proof}
This is a direct corollary of Lemma \ref{l:[G,I]}.
\end{proof}



To describe the sets $[A,[A,A]],[A,[A,[A,A]]],\ldots$, we will formally define them as elements of a series.

\begin{de}\label{d:Pi_n} In an algebra $\bA=(A,F)$, let $\Pi_{1}A=A$, and for $n\geq 2$ let
\begin{equation*}
\begin{aligned}
\Pi_{n} A=[A,\Pi_{n-1} A].
\end{aligned}
\end{equation*}
\end{de}

Note that $\Pi_{2}A=\bigcup_{f\in F}f(A,\dots,A)$. We observe that $\Pi_{n} A$ is an ideal for all $n\in\N$. The next statement follows directly from Definition \ref{d:Pi_n} and Lemma \ref{l:[G,I]}.

\begin{lem}\label{l:pi_n} In an algebra $\bA=(A,F)$, is $\Pi_{1}A=A$, and for $n\geq 2$ we have
\begin{equation*}
\begin{aligned}
\Pi_{n} A=\bigcup_{f\in F}(f(\Pi_{n-1}A,A,\dots,A)&\cup f(A,\Pi_{n-1}A,\dots,A)\cup\ldots\cup\\
&\cup f(A,\dots,A,\Pi_{n-1}A)).
\end{aligned}
\end{equation*}
\end{lem}

Having described the sets which determine whether an algebra $\bA$ with zero is $n$-nilpotent (see Proposition \ref{p:n-nil}), we are almost ready to prove another equivalent condition for this property. All we now need is a way to describe the complexity of a polynomial term in terms of braces. We will define the depth of a polynomial term in the same manner as article \cite{MudrAichOpr} defines the height of a polynomial term (for technical reasons, our definition differs in the case of constants).

\begin{de}\label{d:depth} Let $\bA=(A,F)$ be an algebra with zero. The depth of a non-zero polynomial term $p\in\Pol(A)$ is defined by
    $$\dep(p)=\begin{cases}
                    0 & \text{if $p=x$, where $x$ is a variable or a}\\
                    & \text{constant,} \\
                    1+\max\{\dep(p_{1}),\dots,\dep(p_{n})\} & \text{if $p=f(p_{1},\dots,p_{n})$, where $p_{1},\dots,p_{n}$}\\
                    &\text{$\in\Pol(A)$ and $f\in F$ of arity $n$.}
                \end{cases}$$
\end{de}

\begin{exa} The depth of the polynomial term $a\cdot((x_{1}\cdot b)\cdot x_{2})$ is $3$, while the depth of the polynomial term $((a\cdot (x_{1}\cdot x_{2}))\cdot b)\cdot((x_{1}\cdot x_{3})\cdot c)$ is $4$. Note that the depth of a polynomial term in a groupoid is the maximal number of open braces at any point in the polynomial, increased by $1$.
\end{exa}

Finally, we may prove the main result of this section.

\begin{thm}\label{t:n-nilpot}
    Let $\bA$ be an algebra with zero $0$, and let $n\in\mathbb{N}$. Then $\bA$ is $n$-nilpotent if and only if all polynomials whose polynomial terms are of depth at least $n$ are vanishing polynomials.
\end{thm}
\begin{proof}
 We know that $\bA$ is $n$-nilpotent if and only if $\underbrace{[A,[A,\ldots[}_{n}A,A]\ldots]]=\{0\}$, by Proposition \ref{p:n-nil}. Since $\underbrace{[A,[A,\ldots[}_{n}A,A]\ldots]]=\Pi_{n+1}A$ by Definition \ref{d:Pi_n}, our aim is to prove that for $n\in\mathbb{N}$ the following holds
    \begin{equation}\label{e:nilp}
    \{p(A^{m}):m\in\mathbb{N},\ p\in\Pol^{*}_{m}(A),\ \dep(p)\geq n\}=\Pi_{n+1}A.
    \end{equation}

    We will prove $(\supseteq)$ using induction on $n$. For the basis $(n=1)$, note that for any $f\in F$ of arity $k$, the polynomial $f\tup{x}{1}{k}$ has depth $1$. Therefore, $f(A,\dots,A)\subseteq \{p(A^{m}):m\in\mathbb{N},\ p\in\Pol^{*}_{m}(A),\ \dep(p)\geq 1\}$. From Lemma \ref{l:G^2} follows that we have $\Pi_{2}A\subseteq \{p(A^{m}):m\in\mathbb{N},\ p\in\Pol^{*}_{m}(A),\ \dep(p)\geq 1\}$, so the basis holds. Now, let $n\geq 2$ and suppose that the inclusion $(\supseteq)$ in \eqref{e:nilp} holds for $n$. Let $f\in F$ with arity $k\geq 2$. For the induction step, it suffices to prove the inclusion $f(\Pi_{n} A,A,\dots,A)\subseteq \{p(A^{m}):m\in\mathbb{N},\ p\in\Pol^{*}_{m}(A),\ \dep(p)\geq n\}$, by Lemma \ref{l:pi_n}. Let  $f\tup{a}{1}{k}\in f(\Pi_{n} A,A,\dots,A)$, where $a_1\in\Pi_{n} A$ and $a_2,\dots,a_k\in A$. By the inductive hypothesis, we know that $\Pi_{n} A\subseteq\{p(A^{m}):m\in\mathbb{N},\ p\in\Pol^{*}_{m}(A),\ \dep(p)\geq n-1\}$. Hence, there exists a $q\in\Pol^{*}_{j}(A)$ for some $j\in\mathbb{N}$ such that $ \dep(q)\geq n-1$ and $b_1,\dots,b_j\in A$ with $a_1=q\tup{b}{1}{j}$. We define $p\tup{x}{1}{j+k-1}=f(q\tup{x}{1}{j}, x_{j+1},\dots,x_{j+k-1})\in\Pol^{*}_{j+k-1}(A)$, and we have $d(p)\geq n$, by Definition \ref{d:depth}. Thus,
    $$
    \begin{aligned}
    f\tup{a}{1}{k}&=f(q\tup{b}{1}{j},a_2,\dots,a_k)\\
    &\in\{p(A^{m}):m\in\mathbb{N},\ p\in\Pol^{*}_{m}(A),\ \dep(p)\geq n\}.
    \end{aligned}
    $$

    $(\subseteq)$ Again, we prove the inclusion by induction on $n$. Suppose $n=1$ and let $ p\in\Pol^{*}_{m}(A)$, $m\in\N$, with $\dep(p)\geq 1$. Then we have $p=f\tup{q}{1}{k}$ for a functional symbol $f$ of arity $k\geq 2$ and polynomials $q_1,\dots,q_k$. Since $q_1(A^m),\dots,q_k(A^m)\subseteq A$, we have
    $$p(A^m)=f(q_1(A^m),\dots,q_1(A^m))\subseteq f(A,\dots,A)\subseteq \Pi_{2}A.$$
    Now, let $n\geq 2$ and suppose that the inclusion $(\subseteq)$ in \eqref{e:nilp} is true for $n-1$. For $ p\in\Pol^{*}_{m}(A)$, $m\in\N$, such that $\dep(p)\geq n$ we have $p=f\tup{q}{1}{k}$ for a functional symbol $f$ of arity $k\geq 2$ and polynomials $q_1,\dots,q_k$. Here, at least one of $q_1,\dots,q_k$ has depth at least $n-1$, by Definition \ref{d:depth}. Without loss of generality, assume $\dep(q_1)\geq n-1$. By the induction hypothesis, we have $q_1(A^m)\subseteq \Pi_{n}A$. Hence, using $q_2(A^m),\dots,q_k(A^m)\subseteq A$, we obtain $p(A^m)=f(q_1(A^m),q_2(A^m),\dots,q_k(A^m))\subseteq f(\Pi_{n}A,A,\dots,A)\subseteq \Pi_{n+1}A$, by Lemma \ref{l:pi_n}. \qedhere

\end{proof}

We will conclude this section with an observation about the relationship between $2$-supernilpotency and $2$-nilpotency.

\begin{pro}
    Let $\bA$ be an algebra with zero such that all nonconstant fundamental operations are binary. Then, $\bA$ is 2-nilpotent if and only if it is 2-supernilpotent.
\end{pro}
\begin{proof}
   From Proposition \ref{supernilpnilp}, we know that $\bA$ being $2$-supernilpotent implies it is $2$-nilpotent as well. For the reverse implication, observe that any polynomial term with three or more variables has at least $2$ instances of fundamental operations. Hence, it has depth at least 2. Therefore, by Theorem \ref{t:n-nilpot} we have $p(A^{3})=\{0\}$ for all $p\in\Pol^{*}_{3}(A)$ and $\bA$ is $2$-supernilpotent by Theorem \ref{p:supernil}. 
\end{proof}

However, for $n\geq 3$, $n$-nilpotency and $n$-supernilpotency do not coincide, as we prove in the next section.

\section{Applications in groupoids with zero}

Due to the great number of cases and the complexity of polynomials in the general case, the problem of investigating solvability is difficult to tackle in an arbitrary algebra with zero. For this reason, we will focus on a specific type of such algebras.
Groupoids with zero are the simplest examples of algebras with zero. In the first part of this section, we apply the statements from the previous sections to obtain a characterisation of abelianness, as well as to investigate connections between $n$-nilpotency and $n$-supernilpotency. In the second part, we provide a characterisation for $n$-solvability.

\begin{de}
An algebra $\bG=(G,\cdot,0)$ is a \emph{grupoid with zero} if $(G,\cdot)$ is a grupoid and $0\in G$ is a constant such that $0\cdot g = g\cdot 0=0$ for all $g\in G$.
\end{de}

A groupoid $(G,\cdot,0)$ is a \textit{$0$-groupoid}, if $a\cdot b=0$ for any $a,b\in G$. Since abelian is the same as $1$-supernilpotent, using Proposition \ref{p:supernil}, we obtain the following statement.

\begin{cor}
    Let $\bG$ be a groupoid with zero. Then, $\bG$ is abelian if and only if it is a $0$-groupoid.
\end{cor}

As usual, in a groupoid $(G,\cdot)$, for any two subsets $X$ and $Y$ of $G$ we denote $X\cdot Y=\{xy \mid x\in X,y\in Y\}$, and for any $g\in G$ we denote $gX=\{gx \mid x\in X\}$ and $Xg=\{xg \mid x\in X\}$.

\begin{pro}\label{3-supernilp}
    Let $\bG$ be a groupoid with zero $0$. Then, $\bG$ is 3-supernilpotent if and only if $\bG$ is 3-nilpotent and $(G^{2})^{2}=\{0\}$.
\end{pro}
\begin{proof}  $(\Rightarrow)$ From Proposition \ref{supernilpnilp} we know that $\bG$ is $3$-nilpotent, and from Theorem \ref{p:supernil} we may conclude that the polynomial $(xy)(zt)\in\Pol^{*}_{4}(G)$ is vanishing. The latter implies $(G^{2})^{2}=\{0\}$. \newline
$(\Leftarrow)$ Suppose that $\bG$ is 3-nilpotent and $(G^{2})^{2}=\{0\}$. By Theorem \ref{t:n-nilpot}, the first part means that
$\{p(G^{m}):m\in\mathbb{N},p\in\Pol^{*}_{m}(G), \dep(p)\geq 3\}=\{0\}.$
Now, let $p\in\Pol^{*}_{4}(G)$ and choose an arbitrary subterm of the form $(u\cdot v)$ in it (where both \(u\) and \(v\) can be a variable or a constant). Then $(u\cdot v)$ is contained in a subterm $s$ that can be classified into one of the following categories, according to its form:
$$
\begin{aligned}
(u\cdot v)\cdot(p_{1}\cdot p_{2}) &\text{\ \ or\ \ }(p_{1}\cdot p_{2})\cdot(u\cdot v) \text{\ \ or\ \ }\\
p_{1}(p_{2}(u\cdot v)) \text{\ \ or\ \ }p_{1}((u\cdot v)p_{2}) &\text{\ \ or\ \ } (p_{1}(u\cdot v))p_{2} \text{\ \ or\ \ } ((u\cdot v)p_{1})p_{2},
\end{aligned}
$$
where $p_{1}$ and $p_{2}$ are non-empty subterms. Note that the value of the first two belongs to $(G^{2})^{2}$, while the rest have depth at least $3$. Thus, the value of the subterm $s$ is $0$ for any evaluation of $p$. Now, Proposition \ref{polynomialzero} implies that the value of $p$ is always $0$. Since $p$ was chosen as an arbitrary polynomial from $\Pol^{*}_{4}(G)$, we have proved that all polynomials from $\Pol^{*}_{4}(G)$ are vanishing. So, by Theorem \ref{p:supernil}, $\bG$ is 3-supernilpotent.
\end{proof}

Note that the 'difference' between $n$-nilpotency and $n$-supernilpotency lies in the polynomials whose depth $d$ satisfies $\log_{2}(n+1)\leq d<n$. Since the number and complexity of such polynomials increase with $n$, for $n\geq 4$ it becomes too cumbersome for an elegant statement.

\begin{exa}\label{counterexample} For $n\geq 3$, we construct a groupoid $\mathbf{H}_{n}$ which is $n$-nilpotent, but not $n$-supernilpotent. Let $H_{n}=\{0\}\cup\{a_{i}:1\leq i\leq 2^{n}-1\}$ and $\mathbf{H}_{n}=(H_{n},\cdot)$, with $\cdot$ defined by the following rules: for $x,y\in H_{n}$, we have $x\cdot y=0$, except in the following cases:
$$a_{2i}\cdot a_{2i+1}=a_{i}, \text{ for } 1\leq i\leq 2^{n-1}-1.$$
Then, we have
$$
\begin{aligned}
    a_{2^{n-1}}\cdot a_{2^{n-1}+1}&=a_{2^{n-2}}\\
    a_{2^{n-1}+2}\cdot a_{2^{n-1}+3}&=a_{2^{n-2}+1}\\
    &\vdots\\
    a_{2^{n}-4}\cdot a_{2^{n}-3}&=a_{2^{n-1}-2}\\
    a_{2^{n}-2}\cdot a_{2^{n}-1}&=a_{2^{n-1}-1}.
\end{aligned}
$$
If we multiply the pairs of terms on the left-hand side two by two, we obtain
$$
\begin{aligned}
    (a_{2^{n-1}}\cdot a_{2^{n-1}+1})\cdot (a_{2^{n-1}+2}\cdot a_{2^{n-1}+3})&=a_{2^{n-3}}\\
    &\vdots\\
    (a_{2^{n}-4}\cdot a_{2^{n}-3})\cdot(a_{2^{n}-2}\cdot a_{2^{n}-1})&=a_{2^{n-2}-1}.\\
\end{aligned}
$$
Continuing the process of multiplying terms on the left two by two iteratively, after $n-1$ steps, we obtain a single term
whose value is $a_{1}$. In Figure \ref{f:H4}, we have illustrated the process for $\mathbf{H}_{4}$. Thus, we have a term of depth $n-1$ with $(2^{n}-1)-(2^{n-1}-1)=2^{n-1}$ elements, whose value is non-zero. For $n\geq 3$, we have $2^{n-1}\geq n+1$, so by Theorem \ref{p:supernil}, the groupoid $\mathbf{H}_{n}$ is not $n$-supernilpotent. However, we will prove that it is $n$-nilpotent. First, we prove that in $\mathbf{H}_{n}$ the result of any product of depth $i$, for $1\leq i\leq n$, lies in the set $\{0,a_{1},\ldots,a_{2^{n-i}-1}\}$ (for $i=n$ the set will simply be $\{0\}$). We use mathematical induction on $i$. For the basis, note that the set of all elements that can be obtained as a product of two elements in $\mathbf{H}_{n}$ is exactly $\{0,a_{1},\ldots,a_{2^{n-1}-1}\}$. Suppose now that any product of depth $i$ lies in the set $\{0,a_{1},\ldots,a_{2^{n-i}-1}\}$. Note that any product of depth $i+1$ necessarily contains a subproduct of depth $i$. Thus, by the inductive hypothesis, the result of a product of depth $i+1$ can be obtained as a product of an element from $\{0,a_{1},\ldots,a_{2^{n-i}-1}\}$ and another element from $H_{n}$. The set of all such elements is $\{0,a_{1},\ldots,a_{2^{n-(i+1)}-1}\}$. Therefore, we have proved that the set of all results of products of depth $n$ is $\{0\}$, so Theorem \ref{t:n-nilpot} implies that $\mathbf{H}_{n}$ is $n$-nilpotent.
\end{exa}

\begin{figure}
  \centering
     \begin{tikzpicture}[auto,
     block/.style ={circle, draw=black, text width=1.75em,align=center}
      ]
      \draw (-0.25,3) node[block] (1) {$\cdot$};
      \draw (-3.25,1.5) node[block] (2) {$\cdot$};
      \draw (2.75,1.5) node[block] (3) {$\cdot$};
      \draw (-4.75,0) node[block] (4) {$\cdot$};
      \draw (-1.75,0) node[block] (5) {$\cdot$};
      \draw (1.25,0) node[block] (6) {$\cdot$};
      \draw (4.25,0) node[block] (7) {$\cdot$};
      \draw (-5.5,-1.5) node[block] (8) {$a_{8}$};
      \draw (-4,-1.5) node[block] (9) {$a_{9}$};
            \draw (-2.5,-1.5) node[block] (10) {$a_{10}$};
      \draw (-1,-1.5) node[block] (11) {$a_{11}$};
            \draw (5.0,-1.5) node[block] (12) {$a_{15}$};
      \draw (3.5,-1.5) node[block] (13) {$a_{14}$};
            \draw (2.0,-1.5) node[block] (14) {$a_{13}$};
      \draw (0.5,-1.5) node[block] (15) {$a_{12}$};
      \draw (1)--(2)--(4)--(8)
            (4)--(9)
            (2)--(5)--(10)
            (3)--(7)--(13)
            (1)--(3)--(6)--(15)
            (6)--(14)
            (7)--(12)
            (5)--(11);
 \end{tikzpicture}
 \caption{The only non-zero product of depth $3$ in $\mathbf{H}_{4}$.}
 \label{f:H4}
\end{figure}
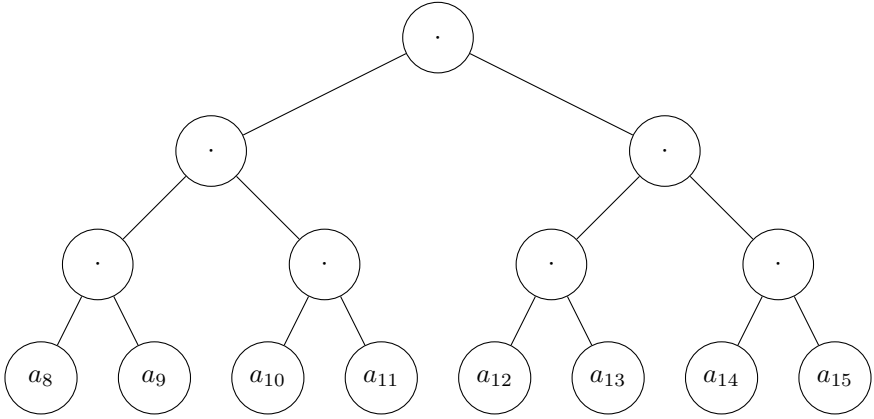

In Proposition \ref{t:n-sol}, we state the condition for $n$-solvability of an algebra with zero in terms of the set $[A]^{(n)}$. To provide a precise (inductive) description of the set $[G]^{(n)}$ and a more convenient characterisation of solvability for a groupoid with zero, we introduce two series of sets.

\begin{de} \label{d:Cseries}
In a groupoid $\bG$, for a set $X\subseteq G$, we define
\begin{equation*}
\begin{aligned}
    C_{0}(X)&=X\\
    C_{n+1}(X)&=G\cdot C_{n}(X) \cup C_{n}(X)\cdot G, \text{ for }n\geq 0.
\end{aligned}
\end{equation*}
\end{de}

Note that $ C_{1}(X)=G\cdot X \cup X\cdot G$ for any set $X$, and therefore we have the following statement.

\begin{pro}\label{inductiveC}
In a groupoid $\bG$, for a set $X\subseteq G$, we have $C_{n+1}(X)=C_1(C_n(X))$ for all $n\in\N_0$.
\end{pro}

\begin{de} \label{d:S}
In a groupoid $\bG$, we define
\begin{equation*}
\begin{aligned}
    S_{1}G&=G\\
    S_{2}G&=G^{2}\\
    S_{n+1}G&=\bigcup_{k\in\mathbb{N}_0}C_{k}((S_{n}G)\cdot(S_{n}G)), \text{ for }n\geq2.
\end{aligned}
\end{equation*}
\end{de}

To prove Proposition \ref{p:nsolGroupoid}, which provides a criterion for checking $n$-solvability, we need to prove the following two statements about the properties of these two series.

\begin{lem}\label{CSnG}
 In a groupoid $\bG$, we have $C_m(S_nG)\subseteq S_nG$ for any $n\in\N$ and any $m\in\N_0$.
\end{lem}

\begin{proof}
For $n=1$, we prove the inclusion directly using the fact that $C_m(X)\subseteq G$ for all $X\subseteq G$ and all $m\in\mathbb{N}$. Let $n\geq 2$. We use induction on $m$. For $m=0$ we have $C_0(S_nG)=S_nG$, by Definition \ref{d:Cseries}. Let $m\geq 1$ and suppose that $C_m(S_nG)\subseteq S_nG$. For $n=2$, from Lemma \ref{inductiveC}, inductive hypothesis, and Definitions \ref{d:Cseries} and \ref{d:S}, we have $C_{m+1}(S_2G)=C_1(C_m(S_2G))\subseteq C_1(S_2G)=C_1(G^2)=G\cdot G^2\cup G^2\cdot G\subseteq G\cdot G=G^2=S_2G$, using the fact $G^2\subseteq G$.
Finally, for $n\geq 3$, using the same statements and definitions, we obtain
$$
\begin{aligned}
C_{m+1}(S_nG)&=C_1(C_m(S_nG))\subseteq C_1(S_nG)\\
&=C_1(\cup_{k\in\N_0}C_k((S_{n-1}G)\cdot(S_{n-1}G))\\
&=\cup_{k\in\N_0}C_1(C_k((S_{n-1}G)\cdot (S_{n-1}G)))\\
&=\cup_{k\in\N_0}C_{k+1}((S_{n-1}G)\cdot(S_{n-1}G))\subseteq S_nG.
\end{aligned}
$$
\end{proof}

\begin{cor}\label{SnGideal}
  Let $n\in\N$. For a groupoid $\bG$, the set $S_nG$ is an ideal of $\bG$.
\end{cor}

\begin{proof}
Obviously, $\emptyset\neq S_nG\subseteq G$. For $g\in G$ and $a\in S_nG$ we have $ag,ga\in G(S_nG)\cup(S_nG)G=C_1(S_nG)\subseteq S_nG$ by Lemma \ref{CSnG}.
\end{proof}

Finally, we may prove the following statement.
\begin{pro}\label{p:nsolGroupoid}
Let $\bG$ be a groupoid with zero $0$ and let $n\geq 1$. Then, $\bG$ is $n$-solvable if and only if $S_{n+1}G=\{0\}$.
\end{pro}
\begin{proof} We will prove by induction on $n$ that $[G]^{(n)}=S_{n+1}G$ for $n\in\mathbb{N}$. Then the statement follows directly from Theorem \ref{t:n-sol}. For the basis, note that Lemma \ref{l:G^2} and Definition \ref{d:S} imply $[G]^{(1)}=[G,G]=G^{2}=S_{2}G$. Now, suppose that $[G]^{(n)}=S_{n+1}G$ holds for some $n\geq 1$ and let us prove the statement for $n+1$. By Definition \ref{d:P_n} and the inductive hypothesis, we have
$$
\begin{aligned}
\relax [G]^{(n+1)}&=[[G]^{(n)},[G]^{(n)}]\\
&=[S_{n+1}G,S_{n+1}G]=\{p(g_{1},g_{2}):p\in\Pol^{*}_{2}(G),\ g_{1},g_{2}\in S_{n+1}G\}.
\end{aligned}
$$
Since $S_{n+2}G=\cup_{k\in\mathbb{N}_{0}}C_{k}((S_{n+1}G)\cdot(S_{n+1}G))$, our goal is to prove
\begin{equation*}
\cup_{k\in\mathbb{N}_{0}}C_{k}((S_{n+1}G)\cdot(S_{n+1}G))=\{p(g_{1},g_{2}):p\in\Pol^{*}_{2}(G),\ g_{1},g_{2}\in S_{n+1}G\}.
\end{equation*}

To show $(\subseteq)$, we prove that
$$
C_{k}((S_{n+1}G)\cdot(S_{n+1}G))\subseteq\{p(g_{1},g_{2}):p\in\Pol^{*}_{2}(G),\ g_{1},g_{2}\in S_{n+1}G\},
$$
for all $k\in\N_0$, by induction on $k$. Since $x\cdot y\in\Pol^{*}_{2}(G)$, for any $g\in C_{0}((S_{n+1}G)\cdot(S_{n+1}G))=(S_{n+1}G)\cdot(S_{n+1}G)$ we have $g=g_{1}\cdot g_{2}$ for $g_{1},g_{2}\in S_{n+1}G$, so $g\in\{p(g_{1},g_{2}):p\in\Pol^{*}_{2}(G),\ g_{1},g_{2}\in S_{n+1}G\}$. Therefore, the base case $(k=0)$ is true. Now, suppose that the statement is true for some $k\in\N_0$ and let us prove it for $k+1$.  By Proposition \ref{inductiveC}, we have $C_{k+1}((S_{n+1}G)\cdot(S_{n+1}G))=C_1(C_{k}((S_{n+1}G)\cdot(S_{n+1}G)))$. Hence, $C_{k+1}((S_{n+1}G)\cdot(S_{n+1}G))\subseteq C_1(\{p(g_{1},g_{2}):p\in\Pol^{*}_{2}(G),\ g_{1},g_{2}\in S_{n+1}G\})$, by the inductive hypothesis. Note that for any $f(x,y)\in\Pol^{*}_{2}(G)$ and for any $a\in G$, we have $af(x,y), f(x,y)a\in\Pol^{*}_{2}(G)$. Therefore
$$
\begin{aligned}
&G\cdot\{p(g_{1},g_{2}):p\in\Pol^{*}_{2}(G),\ g_{1},g_{2}\in S_{n+1}G\}\ \cup \\
&\{p(g_{1},g_{2}):p\in\Pol^{*}_{2}(G),\ g_{1},g_{2}\in S_{n+1}G\}\cdot G\\
&\subseteq\{p(g_{1},g_{2}):p\in\Pol^{*}_{2}(G),\ g_{1},g_{2}\in S_{n+1}G\}.
\end{aligned}
$$
Hence, we obtain $C_1(\{p(g_{1},g_{2}):p\in\Pol^{*}_{2}(G),\ g_{1},g_{2}\in S_{n+1}G\})\subseteq\{p(g_{1},g_{2}):p\in\Pol^{*}_{2}(G),\ g_{1},g_{2}\in S_{n+1}G\}$. Thus, we have proved the induction step.

$(\supseteq)$ We show this by induction on $m$, the number of instances of the symbol $\cdot$ in the polynomial term of $p\in\Pol^{*}_{2}(G)$. For $m=1$, we only have the polynomial term $p(x,y)=x\cdot y\in\Pol^{*}_{2}(G)$ and hence $p(g_{1},g_{2})=g_1g_2\in(S_{n+1}G)\cdot(S_{n+1}G)=C_{0}((S_{n+1}G)\cdot(S_{n+1}G))$ for $g_{1},g_{2}\in S_{n+1}G$. Suppose that the statement holds for any polynomial term of complexity smaller than or equal to $m$, and let us prove it for $p\in\Pol^{*}_{2}(G)$ whose term has $m+1$ instances of $\cdot$. Then $p(x,y)=q(x,y)r(x,y)$, where $q$ and $r$ have at most $m$ occurrences of $\cdot$ in their polynomial terms. We will divide the proof into three cases. If $q(x,y)\in\Pol^{*}_{2}(G)$, then there exists a $k\in\N_0$ such that $q(g_1,g_2)\in C_k((S_{n+1}G)\cdot(S_{n+1}G))$ by the induction hypothesis. Now,
$$
\begin{aligned}
p(g_1,g_2)=q(g_1,g_2)r(g_1,g_2)&\in C_k((S_{n+1}G)\cdot(S_{n+1}G))r(g_1,g_2)\\
&\subseteq C_k((S_{n+1}G)\cdot(S_{n+1}G))G\\
&\subseteq C_{k+1}((S_{n+1}G)\cdot(S_{n+1}G)).
\end{aligned}$$
The case where $r(x,y)\in\Pol^{*}_{2}(G)$ is analogous. So, the remaining case is $q(x,y),r(x,y)\not\in\Pol^{*}_{2}(G)$. Without loss of generality, we assume that $x$ appears in the polynomial term of $q$ and $y$ appears in the polynomial term of $r$. Now, $q(g_1,g_2),r(g_1,g_2)\in S_{n+1}G$ by Proposition \ref{idealsandpolynomials}, because $g_1,g_2\in S_{n+1}G$ and $S_{n+1}G$ is an ideal by Corollary \ref{SnGideal} . Hence $p(g_1,g_2)\in (S_{n+1}G)\cdot(S_{n+1}G)\subseteq\cup_{k\in\mathbb{N}}C_{k}((S_{n+1}G)(S_{n+1}G))$. Thus, we have proved the inclusion in all three cases, which concludes the proof of the induction step.\qedhere
\end{proof}

Using the previous result, we will prove a more nuanced version of Proposition \ref{3-supernilp}.

\begin{cor}
Let $\bG$ be a groupoid with zero. Then $\bG$ is 3-supernilpotent if and only if $\bG$ is 3-nilpotent and $2$-solvable.
\end{cor}

\begin{proof}
From Proposition \ref{3-supernilp} we know that $\bG$ is $3$-supernilpotent if and only if it is $3$-nilpotent and $(G^2)^2=\{0\}$. By Definition \ref{d:S} we have $(S_2G)\cdot(S_2G)=\{0\}$ and, therefore, $C_k((S_2G)\cdot(S_2G))=\{0\}$ for all $k\in\N_0$. Thus, $(G^2)^2=\{0\}$ holds if and only if $S_3G=\{0\}$, which is equivalent to $\bG$ being $2$-solvable, by Proposition \ref{p:nsolGroupoid}.
\end{proof}

We will provide a more operable description of the set $S_{n+1} G$, which will be the foundation for the main result of this section, Theorem \ref{t:nsolGroupoid}. To do this, we need to select polynomial terms of a specific form and examine some of their properties.



\begin{de}\label{d:round_pol} We say that a polynomial term of form $x\cdot y$ is a \textit{round polynomial term of order 1} (the variables may coincide). A polynomial term in a groupoid $\mathbf{G}$ is \textit{round of order} $k\geq 2$ if it can be obtained by finitely many applications of the following two rules:
\begin{enumerate}
    \item if $p$ is a round polynomial term of order $k\geq 2$, then $(c\cdot p)$ and $(p\cdot c)$ are round polynomial terms of order $k$ for any constant $c$;
    \item if $p$ and $q$ are round polynomial terms of order $k\geq 1$, then $(p\cdot q)$ is a round polynomial term of order $k+1$.
\end{enumerate}
As usual, we omit the outer braces of a term for the sake of simplicity. In addition, we omit the operation symbol unless it provides better readability of the term.
\end{de}

\begin{exa}
    The polynomial term $a((b((x_{1}x_{2})(x_{3}x_{4})))c)$ is a round polynomial term of order $2$ and
    $\Big{(}(a((x_{1}x_{3})(x_{2}x_{4})))\cdot(((x_{5}x_{7})(x_{8}x_{6}))c)\Big{)}b$
    is a round polynomial term of order $3$.
\end{exa}

\begin{pro} A round polynomial term of order $n\geq 1$ contains $2^{n}$ variables (counting the number of occurrences of variables).
\end{pro}
\begin{proof}
    The statement is clearly true for $n=1$. It is easily seen that $(xy)(zt)$ is the only round polynomial term of order $2$, so the statement also holds for $n=2$. For $n\geq 3$, suppose that any round polynomial term of order $n-1$ contains $2^{n-1}$ variables. Now, note that a round polynomial term of order $n$ is obtained by combining two round polynomial terms of order $n-1$ using rule (2), followed by a finite number (possibly zero) of applications of rule (1). Since the polynomial terms of order $n-1$ contain $2^{n-1}$ variables and rule (1) does not change the number of variables, the resulting polynomial term contains $2^{n}$ variables.
\end{proof}

\begin{pro}\label{p:qRoundPol} Let $q_1(x_1,x_2)=x_1\cdot x_2$, and for $k\geq 2$
$$q_k\tup{x}{1}{2^k}=q_{k-1}\tup{x}{1}{2^{k-1}}\cdot q_{k-1}\tup{x}{2^{k-1}+1}{2^k}.$$
Then, for $k\geq 1$, $q_k\tup{x}{1}{2^k}$ contains $2^k-1$ instances of $\cdot$ and is the smallest round polynomial term of order $k$ (in terms of the number of instances of $\cdot$), and we have $q_k(G^{2^k})\subseteq S_{k+1}G$.
\end{pro}
\begin{proof} We prove the statements using induction. The base cases will be $k=1$ and $k=2$. The first two statements are true in both cases, since $q_{1}$ and $q_{2}$ are the only round polynomials of orders $1$ and $2$, respectively. For the last statement, we have $q_{1}(G^{2})=G^{2}=S_{2}G$ and $q_{2}(G^{4})=(G^{2})(G^{2})=C_{0}((S_{2}G)\cdot(S_{2}G))\subseteq S_{3}G$ so the basis is true here as well. Suppose that the statements are true for $k\geq 2$, and consider $q_{k+1}$. Since $q_{k+1}=q_{k}\cdot q_{k}$, it contains $2\cdot (2^k-1)+1=2^{k+1}-1$ instances of $\cdot$ by the inductive hypothesis. It is clearly a round polynomial term of order $k+1$ (obtained by rule (2) in Definition \ref{d:round_pol}), and it is the smallest one of that order, since $q_{k}$ is the smallest round polynomial term of order $k$ and rule (2) is the only way to obtain a round polynomial of order $k+1$ from polynomials of smaller order. For the last statement, note that the inductive hypothesis implies
\begin{gather*}
q_{k+1}(G^{2^{k+1}})=q_{k}(G^{2^{k}})\cdot q_{k}(G^{2^{k}})\\
\subseteq (S_{k+1}G)\cdot(S_{k+1}G)=C_{0}((S_{k+1}G)\cdot(S_{k+1}G))\subseteq S_{k+2}G.\qedhere
\end{gather*}
\end{proof}

In order to abbreviate long clauses, for a groupoid $\bG$ and $n\in\mathbb{N}$, we write $\RoundPol^{*}_{2^{n}}(G)$ for the set of all polynomials from $\Pol^{*}_{2^{n}}(G)$ which correspond to some round polynomial term of order $n$.


\begin{pro}\label{p:roundPol} In a groupoid $\bG$ with zero, for $n\geq 1$ we have
    \begin{equation}\label{e:S&roundpol}
       S_{n+1}G=\{p(G^{2^{n}}):\ p\in\RoundPol^{*}_{2^{n}}(G)\}.
    \end{equation}
\end{pro}
\begin{proof} We will prove the statement using mathematical induction on $n$. For $n=1$, from Definition \ref{d:S} we have $S_{2}G=G^{2}$. By Definition \ref{d:round_pol}, any round polynomial term of order $1$ is of the form $x\cdot y$, so we have
$$\{p(G^{2}):\ p\in\RoundPol^{*}_{2}(G) \}=G^{2}.$$
Hence, \eqref{e:S&roundpol} is true for $n=1$.
Let $n\geq 1$ and suppose that \eqref{e:S&roundpol} holds for $S_{n}G$. We want to prove the same for $S_{n+1}G$.

$(\subseteq)$ From Definition \ref{d:S}, we have $S_{n+1}G=\bigcup_{k\in\mathbb{N}}C_{k}((S_{n}G)(S_{n}G))$. Using a 'nested' induction (on $k$) we show that, for all $k\in\mathbb{N}_0$,
\begin{equation}\label{e:k-ind}
\begin{aligned}
C_{k}((S_{n}G)(S_{n}G))\subseteq\{p(G^{2^{n}}):\ p\in\RoundPol^{*}_{2^{n}}(G)\}.
\end{aligned}
\end{equation}
To distinguish between the two, we will call them the $n$-induction and the $k$-induction.
For $k=0$, let $g\in(S_{n}G)(S_{n}G)$. By the inductive hypothesis of the $n$-induction, there exist round polynomial terms $q(x_{1},\ldots,x_{2^{n-1}})$ and $r(y_{1},\ldots,y_{2^{n-1}})$ of order $n-1$ and elements $g_{1},\ldots,g_{2^{n-1}},h_{1},\ldots,h_{2^{n-1}}\in G$ such that $g=q(g_{1},\ldots,g_{2^{n-1}})\cdot r(h_{1},\ldots,h_{2^{n-1}})$. Now,  $q\cdot r$ is a round polynomial term of order $n$, obtained by rule (2) in Definition \ref{d:round_pol}. Therefore, $g$ belongs to the right-hand side of \eqref{e:k-ind}. Hence, the basis of the $k$-induction holds. Now, we suppose that the statement of the $k$-induction is true for some $k\in\N_0$, and we want to prove it for $k+1$. By Proposition \ref{inductiveC} and the inductive hypothesis for the $k$-induction, we have
$$
\begin{aligned}
C_{k+1}((S_{n}G)(S_{n}G))&=C_1(C_k((S_{n}G)(S_{n}G)))\\
&\subseteq C_1(\{p(G^{2^{n}}):\ p\in\RoundPol^{*}_{2^{n}}(G)\}).
\end{aligned}
$$
Note that, for any $a\in G$ and any round polynomial term $p\tup{x}{1}{2^n}$ of order $n$, the polynomial terms $a\cdot p\tup{x}{1}{2^n}, p\tup{x}{1}{2^n}\cdot a$ are also round polynomial terms of order $n$ (by rule (1) in Definition \ref{d:round_pol}) whose corresponding polynomials belong to $\RoundPol^{*}_{2^{n}}(G)$. Hence, we have completed the $k$-induction.

$(\supseteq)$ We will use an induction on the number $m\geq3$ of instances of the symbol $\cdot$ in the round polynomial term of $p\in\RoundPol^{*}_{2^{n}}(G)$. This will be the $m$-induction. By Proposition \ref{p:qRoundPol}, the term $q_{k}$ is the simplest round polynomial term of order $k\geq 1$ in terms of the number of instances of the symbol $\cdot$.  Therefore, for the basis of the $m$-induction we let $m=2^{n}-1$ and consider $q_n\in\RoundPol^{*}_{2^{n}}(G)$. From Proposition \ref{p:qRoundPol} we have $q_n(G^{2^n})\subseteq S_{n+1}G$, so the basis of the $m$-induction is true. Now, suppose that the inclusion $(\supseteq)$ in \eqref{e:S&roundpol} is true for all round polynomial terms of order $n$ having at most $m$ instances of $\cdot$ (where $m\geq 2^{n}-1$). We will prove that the same inclusion holds for a round polynomial term $p$ of order $n$ whose polynomial term has $m+1$ instances of $\cdot$. By Definition \ref{d:round_pol}, we have the following two cases. If $p$ was obtained by rule (2), we have $p(x_1,\dots,x_{2^{n-1}},y_1,\dots,y_{2^{n-1}})=q(x_1,\dots,x_{2^{n-1}})r(y_1,\dots,y_{2^{n-1}})$ for round polynomial terms $q,r$ of order $n-1$ with at most $m$ occurrences of the symbol $\cdot$. Therefore, for any $g_1,\dots,g_{2^n}\in G$ we have $q\tup{g}{1}{2^{n-1}},r\tup{g}{2^{n-1}+1}{2^n}\in S_{n}G$ by the inductive hypothesis of the $m$-induction. Hence, $p\tup{g}{1}{2^n}\in(S_nG)(S_nG)\subseteq S_{n+1}G$. If $p$ was obtained by rule (1), we have $p\tup{x}{1}{2^n}=c\cdot q\tup{x}{1}{2^n}$ (the case $p\tup{x}{1}{2^n}=q\tup{x}{1}{2^n}\cdot c$ is analogous) for a $c\in G$ and a round polynomial term $q\tup{x}{1}{2^n}$ of order $n$ with $m-1$ symbols $\cdot$. Then, for $g_1,\dots,g_{2^n}\in G$, we have $p\tup{g}{1}{2^n}=c\cdot q\tup{g}{1}{2^n}\in cS_{n+1}G$ by the inductive hypothesis of the $m$-induction. By Proposition \ref{idealsandpolynomials} and Lemma \ref{CSnG}, we have $cS_{n+1}G\subseteq GS_{n+1}G\subseteq C_1(S_{n+1}G)\subseteq S_{n+1}G$.
\end{proof}

Thus, directly from Propositions \ref{p:nsolGroupoid} and \ref{p:roundPol} and Definition \ref{vanishing} we obtain a more workable criterion for $n$-solvability of a groupoid with zero.
\begin{thm}\label{t:nsolGroupoid}
    Let $\bG$ be a groupoid with a zero, and let $n\geq 2$. Then $\bG$ is $n$-solvable if and only if all polynomials from $\RoundPol^{*}_{2^{n}}(G)$ are vanishing polynomials.
\end{thm}

\begin{exa}
\label{ex:2sol}
    We will show that, for $n\geq 2$, $n$-solvability implies neither $n$-supernilpotency nor $n$-nilpotency. For a fixed $n\geq 2$, consider the groupoid $\bG_{n}=(\{0,a_{1},\ldots,a_{2n+1}\},\cdot)$ whose operation is defined by the following rules: for $x,y\in G_{n}$ we have $x\cdot y=0$, except in the following cases:
    \begin{equation*}
        a_{1}\cdot a_{2}=a_{3}, \text{\ \ and\ \ } \\
        a_{2k+1}\cdot a_{2k+2}=a_{2k+3} \text{ for } k=1,\ldots,n-1.
    \end{equation*}
    The groupoid $\bG_{2}$ is shown in Table \ref{tab:2solTab}.
    Continuing with the general case, we have $(\cdots((a_{1}\cdot a_{2}) a_{4})\cdots ) a_{2n}=a_{2n+1}$. Since we have a non-zero product of $n+1$ elements, Proposition \ref{p:supernil} implies that $G_{n}$ is not $n$-supernilpotent. Furthermore, the depth of the polynomial term $(\cdots((x_{1}\cdot x_{2}) x_{4})\cdots ) x_{2n}$ is $n$, so Theorem \ref{t:n-nilpot} implies that our groupoid is not n-nilpotent. However, since the product of any two elements in $G_{n}$ belongs to $\{0\}\cup\{a_{2k+1}:k=1,\ldots,n\}$ and we have $x\cdot a_{2k+1}=0$ for any $x\in G_{n}$ and any $k=1,\ldots n$, the polynomial $(x_{1}\cdot x_{2})\cdot(x_{3}\cdot x_{4})$ has a single possible value $0$. From  
    Definition \ref{d:round_pol}, it follows that this polynomial term is a subterm of any round polynomial term of order $n\geq 2$. Thus, Proposition \ref{polynomialzero} and Theorem \ref{t:nsolGroupoid} imply that $G_{n}$ is $n$-solvable for $n\geq 2$.
    \begin{table}
    \centering
    \begin{tabular}{c|cccccc}
         $\cdot$ & $0$ & $a_1$ & $a_2$ & $a_3$ & $a_4$ & $a_5$ \\
         \hline
         $0$ & $0$ & $0$ & $0$ & $0$ & $0$ & $0$ \\
         $a_1$ & $0$ & $0$ & $a_{3}$ & $0$ & $0$ & $0$ \\
         $a_2$ & $0$ & $0$ & $0$ & $0$ & $0$ & $0$ \\
         $a_3$ & $0$ & $0$ & $0$ & $0$ & $a_5$ & $0$ \\
         $a_4$ & $0$ & $0$ & $0$ & $0$ & $0$ & $0$ \\
         $a_5$ & $0$ & $0$ & $0$ & $0$ & $0$ & $0$ \\
    \end{tabular}
    \caption{$G_2$ in Example \ref{ex:2sol}}
    \label{tab:2solTab}
\end{table}
\end{exa}

\subsection*{Acknowledgement}

\section{Declarations}

\subsection*{Ethical approval}
Not applicable.

\subsection*{Competing interests}
Not applicable.

\subsection*{Authors' contributions}
All authors contributed equally.

\subsection*{Availability of data and materials}
Not applicable.

\subsection*{Funding}
The research of both authors was supported by the Ministry of Science, Education and Technological Development of the Republic of Serbia. For the second author, through Grants No. 451-03-33/2026-03/200125 and 451-03-34/2026-03/200125.

\end{document}